\documentclass[11pt]{amsart}   
\usepackage[greek,english]{babel}  
\usepackage{graphicx}
\usepackage{amsmath}
\usepackage{amssymb}
\usepackage{amsmath,amsthm}
\usepackage{verbatim}
\usepackage{sidecap}
\usepackage{tikz}
\usepackage{soul}
\newtheorem{theorem}{Theorem}[section]

\newtheorem{lemma}[theorem]{Lemma}

\newtheorem{proposition}[theorem]{Proposition}

\newtheorem{corollary}[theorem]{Corollary}

\theoremstyle{definition}
\newtheorem{remark}[theorem]{Remark}

\def\ve{\varepsilon}

\title[The spatial $N$-centre problem at positive energies]{The spatial $N$-centre problem: scattering at positive energies}
\author{Alberto Boscaggin, Arthur Bottois and Walter Dambrosio}
\address{Alberto Boscaggin and Walter Dambrosio \newline \indent
Dipartimento di Matematica, Universit\`a di Torino, \newline \indent
Via Carlo Alberto 10, 10123 Torino, Italy \newline \newline \indent
Arthur Bottois \newline \indent
\'{E}cole Centrale de Lyon, \newline \indent
36 avenue Guy de Collongue, 69134 \'{E}cully, France \newline \indent
\medbreak
e-mail addresses: alberto.boscaggin@unito.it, arthur.bottois@ecl13.ec-lyon.fr, walter.dambrosio@unito.it}
\date{}

\begin{document}

\begin{abstract}
For the spatial generalized $N$-centre problem 
$$
\ddot{x} = -\sum_{i=1}^{N} \frac{m_i (x - c_i)}{\vert x - c_i \vert^{\alpha+2}},\qquad
x \in \mathbb{R}^3 \setminus \{c_1,\dots,c_N \},
$$
where $m_i > 0$ and $\alpha \in [1,2)$, we prove the existence of positive energy entire solutions with prescribed scattering angle.
The proof relies on variational arguments, within an approximation procedure via (free-time) boundary value problems. 
A self-contained appendix describing a general strategy to rule out the occurrence of collisions is also included.
\end{abstract}

\date{}

\thanks{{\bf Acknowlegments.} Work partially supported by the 
ERC Advanced Grant 2013 n. 339958
{\it Complex Patterns for Strongly Interacting Dynamical Systems - COMPAT}, by the PRIN-2012-74FYK7 Grant {\it Variational and perturbative aspects of nonlinear differential problems} and by the INDAM-GNAMPA Project \textit{Dinamiche complesse per il problema degli $N$-centri}.}

\maketitle 

\section{Introduction and statement of the main result}\label{sect1}
\setcounter{page}{1}

In this paper, we deal with \emph{positive energy solutions} of the differential equation
\begin{equation}\label{main}
\ddot{x} = -\sum_{i=1}^{N} \frac{m_i (x - c_i)}{\vert x - c_i \vert^{\alpha+2}},\qquad
x \in \mathbb{R}^3 \setminus \{c_1,\dots,c_N \},
\end{equation}
where $c_i$ are fixed vectors in $\mathbb{R}^3$ (with $c_i \neq c_j$ for $i \neq j$), 
$m_i > 0$ for all $i$ and $\alpha \in [1,2)$. Here and in the following, 
energy is meant with respect to the natural Hamiltonian structure of \eqref{main}, namely
\begin{equation}\label{def_V}
\ddot x = \nabla V(x), \qquad \mbox{ where } \quad V(x) = \sum_{i=1}^{N} \frac{m_i}{\vert x - c_i \vert^\alpha};
\end{equation}
accordingly, a solution $x$ is said to have positive energy if $\tfrac12 \vert \dot x(t) \vert^2 - V(x(t)) \equiv H$ for some $H > 0$. 

Equation \eqref{main} has to be interpreted as a generalized version of the $N$-centre problem of Celestial Mechanics, that is, the problem of the motion (in the three-dimensional space) of a test particle $x$ under the attraction of $N$ fixed heavy bodies 
$c_1,\ldots,c_N$. This corresponds to equation \eqref{main} for $\alpha = 1$, while the choices $\alpha \in (1,2)$ allows us to deal with non-Newtonian interactions as well (incidentally, recall that for $\alpha \geq 2$ the singularities satisfy the so-called strong force condition; as already remarked by Poincar\'e, in this case the problem becomes simpler, see \cite{AmbCot93}). It is trivial, but useful, to remark that when $N = 1$ equation \eqref{main} just reduces to the well known (generalized) Kepler problem, while the case $N=2$ is often referred to as Euler-Jacobi problem and is solvable as well (see \cite{Pin17,Whi59}). For $N \geq 3$, on the contrary, the problem turns out to be analitically non-integrable \cite{Bol84,BolKoz17,BolNeg01} and, in spite of its simple-looking structure, very little is known in general. 

As for negative energy solutions, one should expect a mixture of motions on KAM tori and chaotic trajectories; however, the mathematical literature in this direction is still extremely limited and the only contributions we are aware of are available for the corresponding problem in the two-dimensional space \cite{BolNeg03,Dim10,SoaTer12,Yu16}. On the other hand, positive energy solutions enjoy the property of escaping to infinity when exiting a sufficiently large ball (by an easy Lagrange-Jacobi argument, see \eqref{ineq_LJ}) and the typical problem becomes the one of \emph{scattering}, namely (very roughly speaking) studying existence and multiplicity of globally defined solutions interacting, on a finite time interval, with the set of centers and having prescribed asymptotic behavior for $t \to \pm \infty$. 
The crucial reference for this is the remarkable paper \cite{Kna02} by Knauf, using tools of perturbative nature to analyze in detail the structure of the set of scattering solutions to \eqref{main} in the high-energy regime (and for $\alpha = 1$). Therein, a non-collinearity condition on the set of the centres and a related assumption on the scattering angle are also required. We remark that all these restrictions were not needed for the corresponding analysis in the two-dimensional case \cite{KleKna92}, as a further evidence of the substantial difficulties arising when facing the spatial problem.

In our brief note, we deal with the three-dimensional scattering problem, establishing the following easy-reading result. It just provides the mere existence of one positive energy solution (from now on, \emph{hyperbolic solution}) but, on the other hand, it is valid for any choice of the centres, for any (but two) scattering angle and for any positive energy. 

\begin{theorem}\label{theo_main}
Let $N \geq 2$. For any $\xi^-,\xi^+ \in \mathbb{S}^2$ with $\xi^+ \neq \pm\xi^-$ and any $h > 0$, there exists a solution $x: \mathbb{R} \to \mathbb{R}^3 \setminus \{c_1,\dots,c_N \}$ of \eqref{main} with energy $H > 0$ such that
\begin{equation}\label{asym_x}
\lim_{t \to \pm\infty} \frac{x(t)}{\vert x(t) \vert} = \xi^{\pm}. 
\end{equation}
\end{theorem}

For the proof of Theorem \ref{theo_main}, we use a variational approach together with an approximation scheme. More precisely, first for any $R > 0$ large enough a solution $x_R$ of the two-point problem (with free-time and fixed energy)
\begin{equation}\label{bvp}
\ddot x_R = \nabla V(x_R), \qquad \frac12 \vert \dot x_R \vert^2 - V(x_R) = H, \qquad x_R(\pm\omega_R) = R\xi^{\pm}, 
\end{equation}
is found as a min-max critical point of the associated Maupertuis functional; second an entire solution $x_\infty$, having asymptotic directions $\xi^{\pm}$ for $t \to \pm\infty$, is constructed as the limit $x_\infty(t) = \lim_{R \to +\infty} x_R(t)$. It worth pointing out that the excluded situations $\xi^- = -\xi^+$ and $\xi^- = \xi^+$ are due to very different reasons: in the first case, it seems impossible to exclude that the sequence of solution $x_R$ escapes to infinity when $R \to +\infty$; in the second one, it seems impossible to exclude the presence of a collision with the set of centres. 

Blow-up arguments, Morse index estimates and regularization techniques play a crucial role in making our procedure effective. More precisely, the blow analysis takes advantage of arguments previously developed both in \cite{FelTan00} (dealing with a one-center like potential, under strong force type assumptions both at the singularity and at infinity) and in \cite{BosDamPap17,BosDamTer17} (dealing with the generalized $N$-centre problem, at the zero-energy level). On the other hand, the strategy to rule out the occurrence of collisions is inspired by the one in \cite{Tan94,Tan93} but is here sharpened by the use of the classical estimates at collisions by Sperling \cite{Spe69}: all this is carefully presented in a final Appendix, hopefully of independent interest.

We end this introduction with a final remark and an open problem. In the proof of our main result, the assumption $N \geq 2$ plays a role 
(we refer to \cite[Proposition 0.1]{FelTan00} for the scattering analysis when $N = 1$) and the solutions found are indeed believed to interact with the centres, though no explicit estimate is available. It seems reasonable that regarding 
\eqref{main} as a perturbation at infinity of the generalized Kepler problem may lead, at least for sufficiently small $H > 0$, to a distinct hyperbolic solution, having the same asymptotic directions but staying far away from the centres. Unfortunately, we have been unable to prove (or disprove) this conjecture and we leave it as a possibly interesting open question for future investigations. 

\subsection{Plan of the paper}

In the subsequent subsections of this Introduction we fix some notation and we prove some useful technical estimates and results needed throughout the paper.

In Section \ref{sec2} we deal with the Bolza problem \eqref{bvp}.

In Section \ref{sec3} we prove that the approximated solutions found in Section \ref{sec2} converge to an entire hyperbolic solution of \eqref{main} when $R\to +\infty$; moreover, we show that this solution has the desired asymptotic properties.

In a (self-contained) final Appendix, we investigate generalized solutions of a perturbed Kepler problem, collecting arguments used along the proofs of Sections \ref{sec2} and \ref{sec3} to rule out the occurrence of collisions.

\subsection{Notation}

The symbols $x \cdot y$ and $\vert  x \vert$ denote the standard Euclidean product and Euclidean norm on $\mathbb{R}^3$, $B_\rho(x)$ is the open ball of radius $\rho$ centered at $x$. The symbols $\langle u,v \rangle$ and $\Vert u \Vert$ stand for the usual scalar product and the associated norm on the Sobolev space $H^1([a,b];\mathbb{R}^3)$, namely 
$$\langle u,v \rangle = \int_a^b \left(u(t) \cdot v(t) + \dot{u}(t) \cdot \dot{v}(t) \right)\,dt,\qquad
\Vert u \Vert = \left[\int_a^b \left(\vert u(t) \vert^2 + \vert \dot{u}(t) \vert^2 \right)\,dt \right]^{1/2}.$$
Finally, $\text{j}(A)$ is the Morse-index of a self-adjoint bounded linear operator $A$ on an Hilbert space.

\subsection{Technical estimates on the potential}\label{subsect_V}

Let us define 
\begin{equation}\label{def_m_Xi}
\Sigma = \{c_1,\dots,c_N \}, \qquad \Xi = \max_i \vert c_i \vert, \qquad m = \sum_{i=1}^N m_i;
\end{equation}
without loss of generality, we assume henceforth that the center of mass is placed at the origin, namely
\begin{equation}\label{def_origin}
\sum_{i=1}^N m_i c_i = 0.
\end{equation}
Using the above notation, we collect here below
some properties of the potential $V$ (recall the definition given in \eqref{def_V}) near the centers $c_i$ and at infinity.

Precisely, as for the behavior of $V$ near the singularities, for any $i=1,\dots,N$ we write
\begin{equation}\label{V_sing}
V(x) = \frac{m_i}{\alpha \vert x - c_i \vert^\alpha} + \Phi_i(x),
\end{equation}
with $\Phi_i \in \mathcal{C}^\infty(\mathbb{R}^3 \setminus (\Sigma \setminus \{c_i \}))$. From now on, we fix a constant $\delta^* > 0$ so small that 
\begin{equation}\label{deltastar1}
B_{\delta^*}(c_i) \subset B_{\Xi + 1}(0),\; \forall i=1,\dots,N,\qquad
B_{\delta^*}(c_i) \cap B_{\delta^*}(c_j) = \emptyset,\; \forall i \neq j.
\end{equation}
Moreover, we also assume
\begin{equation}\label{deltastar2}
\frac{(2 - \alpha) m_i}{\alpha \vert x - c_i \vert^\alpha} + 2 \Phi_i(x) + \nabla \Phi_i(x) \cdot (x - c_i) \geq 0,\quad
\mbox{ for } 0 < \vert x - c_i \vert \leq \delta^*,
\end{equation}
and
\begin{equation}\label{deltastar3}
V(x) + H \leq \frac{3 m_i}{2 \alpha \vert x - c_i \vert^\alpha},\quad
\mbox{ for } 0 < \vert x - c_i \vert \leq \delta^*, 
\end{equation}
for $i=1,\dots,N$.

On the other hand, dealing with the behavior of $V$ at infinity, we set
\begin{equation}\label{V_inf}
V(x) = \frac{m}{\alpha \vert x \vert^\alpha} + W(x).
\end{equation}
Using \eqref{def_m_Xi}, we can easily see that
$$W(x) = \mathcal{O} \left(\frac{1}{\vert x \vert^{\alpha+2}} \right)\quad
\mbox{ and }\quad
\nabla W(x) = \mathcal{O} \left(\frac{1}{\vert x \vert^{\alpha+3}} \right)\quad
\mbox{ as } \vert x \vert \to +\infty.$$
As a consequence, we can chose constants $C_-,C_+ > 0$ and $K > \Xi + 1$ such that
\begin{equation}\label{K1}
\vert W(x) \vert \leq \frac{C_+}{\vert x \vert^{\alpha+2}}\quad
\mbox{ and }\quad
\vert \nabla W(x) \vert \leq \frac{C_+}{\vert x \vert^{\alpha+3}},\quad \mbox{ for every } \vert x \vert \geq K,
\end{equation}
\begin{equation}\label{K2}
\frac{(2 - \alpha) m}{\alpha \vert x \vert^\alpha} + 2 W(x) + \nabla W(x) \cdot x \geq 0,\quad \mbox{ for every } \vert x \vert \geq K,
\end{equation}
\begin{equation}\label{K3}
\frac{C_-}{\vert x \vert^\alpha} \leq V(x) \leq \frac{C_+}{\vert x \vert^\alpha},\quad \mbox{ for every } \vert x \vert \geq K,
\end{equation}
and
\begin{equation}\label{K4}
\sqrt{\frac{m}{\alpha \vert x \vert^\alpha} + H} - \frac{C_+}{\vert x \vert^{\alpha+2}} \leq \sqrt{V(x) + H} \leq \sqrt{\frac{m}{\alpha \vert x \vert^\alpha} + H} + \frac{C_+}{\vert x \vert^{\alpha+2}}
\end{equation}
for every $\vert x \vert \geq K$. The estimates \eqref{K1}, \eqref{K2} and \eqref{K3} are straightforward, while \eqref{K4} is a consequence of \eqref{K1} and of the elementary inequalities $1 - 2 \vert s \vert \leq \sqrt{1 + s} \leq 1 + \frac{s}{2}$ (valid for $s \geq -1$).

\subsection{Estimating large hyperbolic solutions}\label{subsect_LJ}

In this section we collect some preliminary estimates valid for ``large'' hyperbolic solutions of \eqref{main}. More precisely, we deal with solutions $x: [t_1,t_2] \to \mathbb{R}^3$ of \eqref{main}, with $-\infty \leq t_1 < t_2 \leq +\infty$ (in the case $t_i \in \{\pm\infty \}$, we agree that $t_i$ is not included in the interval of definition of $x$), satisfying the energy relation
\begin{equation}\label{energy}
\frac{1}{2} \vert \dot{x}(t) \vert^2 = V(x(t)) + H,\quad
\mbox{ for every } t \in [t_1,t_2],
\end{equation}
and
\begin{equation}\label{large_sol}
\vert x(t) \vert \geq K,\quad
\mbox{ for every } t \in [t_1,t_2],
\end{equation}
where $K > \Xi + 1$ is the constant fixed in Subsection \ref{subsect_V}. Writing in polar coordinates
\begin{equation}\label{polar_coor}
x(t) = r(t) s(t)
\end{equation}
with $r(t) = \vert x(t) \vert \geq K$ and $s(t) = \tfrac{x(t)}{\vert x(t) \vert} \in \mathbb{S}^2$, the energy relation reads as
\begin{equation}\label{polar_energy}
\dot{r}^2 + r^2 \vert \dot{s} \vert^2 = 2 (V(rs) + H),
\end{equation}
while the differential equation \eqref{main} becomes
\begin{equation}\label{polar_main}
\ddot{r} = r \vert \dot{s} \vert^2 + \nabla V(rs) \cdot s,\qquad
r \ddot{s} = \nabla_{\mathbb{S}^2} V(rs) - r \vert \dot{s} \vert^2 s - 2 \dot{r} \dot{s},
\end{equation}
where $\nabla_{\mathbb{S}^2} V(rs) = \nabla V(rs) - (\nabla V(rs) \cdot s) s$.

Let us define define the moment of inertia
$$I(t) = \frac{1}{2} \vert x(t) \vert^2 = \frac{1}{2} r^2(t)$$
and the angular momentum
$$A(t) = x(t) \wedge \dot{x}(t),$$
for every $t \in [t_1,t_2];$ notice that, in the coordinates \eqref{polar_coor}, we have
\begin{equation}\label{A}
\vert A(t) \vert = r^2(t) \vert \dot{s}(t) \vert. 
\end{equation}
Using the estimates \eqref{K1} and \eqref{K2} it is immediate to prove that
	\begin{equation}\label{ineq_LJ}
	\ddot{I}(t) \geq 2 H > 0\quad \mbox{ for every } t \in [t_1,t_2],
	\end{equation}
and that
	\begin{equation}\label{ineq_dA}
	\vert \dot{A}(t) \vert \leq \frac{C_+}{r^{\alpha + 2}(t)},\quad
	\mbox{ for every } t \in [t_1,t_2],
	\end{equation}
	where $C_+ > 0$ is the constant fixed in Subsection \ref{subsect_V}.

As a first consequence of \eqref{ineq_LJ}, either $r$ is strictly monotone on $[t_1,t_2]$ or 
there exists $t^* \in (t_1,t_2)$ such that $r$ is strictly decreasing on $[t_1,t^*)$ and strictly increasing on $(t^*,t_2]$.

From \eqref{ineq_LJ} and \eqref{ineq_dA} we can also establish the following results, which will be used various times in the paper.

\begin{lemma}\label{cor_t}
	Let $x: [t_1,t_2] \to \mathbb{R}^3$ be a solution of \eqref{main} satisfying \eqref{energy}-\eqref{large_sol} and assume that 
	$r$ is strictly monotone on the whole $[t_1,t_2]$. Then
	\begin{equation}\label{t_est}
	\frac{\vert r(t_2) - r(t_1) \vert}{\sqrt{2 (H + C_+/K^\alpha)}} \leq
	t_2 - t_1 \leq
	\frac{1}{\sqrt{2 H}} \max\{r(t_2),r(t_1) \}.
	\end{equation}
\end{lemma}

\begin{proof}
	We give the proof when $r$ is strictly increasing (the other case being analogous). At first, notice that, in view of the previous discussion, $\dot{r}(t_1) \geq 0$ and $\dot{r}(t) > 0$ for $t \in (t_1,t_2]$. Using the fact that $x$ has energy $H$ and \eqref{K3}, we find
	$$0 < \dot{r}(t) \leq \sqrt{2 (H + C_+/K^\alpha)},\quad
	\mbox{ for every } t \in (t_1,t_2].$$
	Hence
	\begin{align*}
	t_2 - t_1 & = \int_{t_1}^{t_2} \frac{\dot{r}(t)}{\dot{r}(t)}\,dt \geq
	\frac{1}{\sqrt{2 (H + C_+/K^\alpha)}} \int_{t_1}^{t_2} \dot{r}(t)\,dt\\
	& = \frac{r(t_2) - r(t_1)}{\sqrt{2 (H + C_+/K^\alpha)}},
	\end{align*}
	thus proving the estimate from below. 
	On the other hand, using \eqref{ineq_LJ} we find, for $t \in [t_1,t_2]$, 
	$$\dot{I}(t) \geq 2 H (t - t_1).$$
	Integrating on $[t_1,t_2]$, we thus have
	$$\frac{1}{2} r^2(t_2) = I(t_2) \geq H (t_2 - t_1)^2,$$
	giving the estimate from above.
\end{proof}

\begin{lemma}\label{lem_ids}
	Let $x: [t_1,t_2] \to \mathbb{R}^3$ be a solution of \eqref{main} satisfying \eqref{energy}-\eqref{large_sol} and assume that $r$ is strictly increasing on the whole $[t_1,t_2]$. Then, for any $\tau$ with $t_1 < \tau \leq t_2$,
	\begin{equation}\label{ineq_ids}
	\int_{\tau}^{t_2} \vert \dot{s}(t) \vert\,dt \leq
	\frac{C_1 + C_2 r(t_1)}{2 H (\tau - t_1)},
	\end{equation}
	where $C_1,C_2$ are positive constants depending only on $\alpha$, $H$, $C_+$ and $K$.
\end{lemma}

Clearly, a symmetric result can be given when $r$ is strictly decreasing on $[t_1,t_2]$.

\begin{proof}
	At first, we observe that, using \eqref{ineq_LJ} and Lemma \ref{cor_t}, we find
	\begin{align}\label{ids1}
	I(t) \geq \frac{1}{2} + H (t - t_1)^2 \geq H (t - t_1)^2\quad
	\mbox{ for every } t \in [t_1,t_2].
	\end{align}
	Now, we write \eqref{ineq_dA} as
	$$\vert \dot{A}(t) \vert \leq \frac{C_+}{2 I(t)},\quad
	\mbox{ for every } t \in [t_1,t_2].$$
	Recalling \eqref{ids1}, we find, for $t \in [t_1,t_2]$,
	$$\vert \dot{A}(t) \vert \leq \frac{C_+}{1 + 2 H (t - t_1)^2},$$
	so that
	$$\int_{t_1}^{t} \vert \dot{A}(s) \vert\,ds \leq C_1,\quad
	\mbox{ for every } t \in [t_1,t_2],$$
	where
	$$C_1 = \frac{\pi C_+}{2 \sqrt{2 H}}.$$
	Therefore, using the energy relation and \eqref{K3}, for every $t \in [t_1,t_2]$,
	$$\vert A(t) \vert \leq \vert A(t_1) \vert + \int_{t_1}^t \vert \dot{A}(s) \vert\,ds \leq  C_1 + C_2 r(t_1),$$
	where $C_2 = \sqrt{2 (H + C_+/K^\alpha)}$. Recalling \eqref{A} and \eqref{ids1}, we thus find
	$$\vert \dot{s}(t) \vert \leq \frac{C_1 + C_2 r(t_1)}{2 H (t - t_1)^2},\quad
	\mbox{ for every } t \in (t_1,t_2].$$
	Finally, we obtain, for $\tau \in (t_1,t_2]$,
	$$\int_{\tau}^{t_2} \vert \dot{s}(t) \vert\,dt \leq
	\frac{C_1 + C_2 r(t_1)}{2 H} \int_\tau^{t_2} \frac{1}{(t - t_1)^2}\,dt \leq \frac{C_1 + C_2 r(t_1)}{2 H (\tau - t_1)},$$
	and the proof is thus concluded.
\end{proof}

\section{An approximating problem}\label{sec2}

In this section we look for hyperbolic solutions with energy $H > 0$ of the (free-time) fixed-endpoints problem
\begin{equation}\label{syst_xR}
\left\{ \begin{array}{l} \vspace{0.1cm}
\ddot{x}_R = \nabla V(x_R),\\
x_R(\pm\omega_R) = R \xi^\pm,
\end{array} \right.
\end{equation}
with $V$ defined in \eqref{def_V} and $R>K$. Solutions of \eqref{syst_xR} can be seen as approximated solutions of entire hyperbolic solutions of \eqref{main}; the goal of this section is to construct solutions of \eqref{syst_xR} that converge to entire solutions of \eqref{main} as $R\to +\infty$. 
\smallbreak
More precisely, we are going to state and prove the following result. In the statement, we employ the notation
$$\mathcal{A}_{[a,b]}(x) = \int_a^b \left(\frac{1}{2} \vert \dot{x}(t) \vert^2 + V(x(t)) + H \right)\,dt$$
for any $x \in H^1([a,b];\mathbb{R}^3)$. As well known, if $\bar{x}: [a,b] \to \mathbb{R}^3 \setminus \Sigma$ is a (non-collision) solution of $\ddot{x} = \nabla V(x)$, then $\bar{x}$ is a critical point of the (action) functional $\mathcal{A}_{[a,b]}$ on the domain $\{x \in H^1([a,b];\mathbb{R}^3 \setminus \Sigma)\, :\, x(a) = \bar{x}(a), x(b) = \bar{x}(b) \}$.

\begin{theorem}\label{theo_act}
	Let $K > \Xi + 1$ be the constant given in Subsection \ref{subsect_V}. Then, for any $R > K$ and for any $\xi^-,\xi^+ \in \mathbb{S}^2$ with $\xi^+ \neq \pm\xi^-$, there exist $\omega_R > 0$ and a hyperbolic solution $x_R$ with energy $H > 0$ of \eqref{syst_xR} satisfying 
	\begin{equation}\label{est_jA}
	\textnormal{j}(d^2 \mathcal{A}_{[-\omega_R,\omega_R]}(x_R)) \leq 1
	\end{equation}
	and
	\begin{equation}\label{est_act}
	2 \sqrt{2 H} (R - K) \leq \mathcal{A}_{[-\omega_R,\omega_R]}(x_R) \leq 2 \sqrt{2} F_K(R) + M,
	\end{equation}
	where $M > 0$ is a suitable constant not depending on $R$ and
	$$F_K(R) := \int_K^R \sqrt{\frac{m}{\alpha r^\alpha} + H}\,dr,\quad \mbox{ for every }  R>K.$$
\end{theorem}

The most crucial part of Theorem \ref{theo_act}, in view of the rest of the paper, is the fact that the solution $x_R$ satisfies the level estimate \eqref{est_act}. This estimate is fundamental to show the convergence of $x_R$ to an entire solution of \eqref{main}; we observe that it is not fulfilled by minimizing solutions to \eqref{syst_xR}. Hence, as we will see below, $x_R$ has to be found as a critical point with nontrivial (but not greater than one) Morse index, via a suitable min-max procedure.

\smallbreak 
The proof of Theorem \ref{theo_act} follows the same lines of the one of Theorem 4.1 in \cite{BosDamTer17}, where parabolic solutions are concerned. The hyperbolic case is somehow easier, due to presence of a positive energy $H>0$ in the action functional; in what follows we only give a sketch of the main steps, referring to \cite{BosDamTer17} for the missing details.
\smallbreak
Let us define a modified potential $V_\beta$, for $\beta \in [0,1]$, by setting
$$V_\beta(x) = V(x) + \beta \widetilde V(x),\quad x \in \mathbb{R}^3 \setminus \Sigma,$$
where $\widetilde V \in \mathcal{C}^\infty(\mathbb{R}^3 \setminus \Sigma)$ is defined as 
$$\widetilde V(x) = \sum_{i=1}^N \frac{m_i}{\vert x - c_i \vert^2} \Psi(\vert x - c_i \vert^2),$$
with $\Psi \in \mathcal{C}^\infty(\mathbb{R}^+;[0,1])$ a cut-off function such that $\Psi(r) = 1$ if $0 \leq r \leq \delta^*$ and $\Psi(r) = 0$ if $r \geq 2 \delta^*$, with $\delta^* > 0$ given by \eqref{deltastar1}.
We then introduce the Maupertuis functional
$$
\mathcal{M}_\beta(u) = \int_{-1}^1 \vert \dot{u}(t) \vert^2\,dt \int_{-1}^1 \Big(V_\beta(u(t)) + H \Big)\,dt
$$
defined on the Hilbert manifold
$$
\Gamma = \Big\{u \in H^1([-1,1];\mathbb{R}^3 \setminus \Sigma)\, :\, u(\pm 1) = R \xi^\pm \Big\}.
$$
As well-known (see, for instance, \cite[Theorem 4.1]{AmbCot93} and \cite[Appendix B]{SoaTer13}), $\mathcal{M}_\beta$ is smooth and any critical point $u_\beta \in \Gamma$ satisfies, for $t \in [-1,1]$, 
\begin{equation}\label{ubeta}
\ddot{u}_\beta(t) = \omega_\beta^2 \nabla V_\beta(u_\beta(t)),\qquad \frac{1}{2} \vert \dot{u}_\beta(t) \vert^2 = \omega_\beta^2 \Big(V_\beta(u_\beta(t)) + H \Big),
\end{equation}
where 
\begin{equation}\label{def_wbeta}
\omega_\beta = \left(\frac{\int_{-1}^1 \vert \dot{u}_\beta \vert^2}{2 \int_{-1}^1 (V_\beta(u_\beta) + H)} \right)^{1/2}.
\end{equation}
Notice that, since $\xi^+ \neq \xi^-$, $u_\beta$ is not constant: as a consequence, $\omega_\beta > 0$ and the function
\begin{equation}\label{def_xbeta}
x_\beta(t) = u_\beta \left(\frac{t}{\omega_\beta} \right),\quad t \in [-\omega_\beta,\omega_\beta],
\end{equation}
is a hyperbolic solution with energy $H$ of $\ddot{x}_\beta = \nabla V_\beta(x_\beta)$ on the interval $[-\omega_\beta,\omega_\beta]$ and, of course, $x_\beta(\pm\omega_\beta) = R\xi^\pm$.

We will look for critical points of ${\mathcal M}_\beta$ of min-max type; to this aim, for any $h \in \mathcal{C}(\mathbb{S}^1,\Gamma)$ and for $i = 1,2$, set
$$\tilde{h}_i: \mathbb{S}^1 \times [-1,1] \to \mathbb{S}^2,\qquad
(s,t) \mapsto \frac{h(s)(t) - c_i}{\vert h(s)(t) - c_i \vert}.$$
Since $h(s)(\pm 1) = R\xi^\pm$ for any $s \in \mathbb{S}^1$, the map $\tilde{h}_i$ can be identified with a continuous self-map on $\mathbb{S}^2$ and so it has a well-defined degree $\text{deg}_{\mathbb{S}^2}(\tilde{h}_i)$ \cite{GraDug03}. We can thus define the class 
\begin{equation}\label{def_minmax}
\Lambda = \Lambda_{R} = \Big\{h \in \mathcal{C}(\mathbb{S}^1,\Gamma)\, :\, \text{deg}_{\mathbb{S}^2}(\tilde{h}_1) \neq 0 = \text{deg}_{\mathbb{S}^2}(\tilde{h}_2) \Big\} 
\end{equation}
(it is clear that this set is non-empty) and the associated min-max value
\begin{equation}\label{def_minmaxval}
c_\beta = c_{\beta,R} = \inf_{h \in \Lambda} \sup_{s \in \mathbb{S}^1} \mathcal{M}_\beta(h(s)).
\end{equation}
Observing that ${\mathcal M}_\beta$ has good compactness properties (both at infinity and near the singular set, compare with \cite[Lemma 4.2]{BosDamTer17}), it is possible to prove the following result.

\begin{proposition}\label{prop_minmax}
	For any $\beta > 0$, $c_\beta$ is a critical value for the functional $\mathcal{M}_\beta$. In particular, there exists $u_\beta = u_{\beta,R} \in \Gamma$ such that
	$$\mathcal{M}_\beta(u_\beta) = c_\beta,\quad
	\nabla \mathcal{M}_\beta(u_\beta) = 0,\quad
	\textnormal{j} \left(d^2 \mathcal{M}_\beta(u_\beta) \right) \leq 1.$$
\end{proposition}

Now, passing to the limit for $\beta \to 0^+$, following \cite[Sect. 4.2]{BosDamTer17}, we deduce the existence of $\omega_R>0$ and of a generalized solution $x_R\in \mathcal{C}([-\omega_R,\omega_R];\mathbb{R}^3)$ of \eqref{syst_xR}, meaning that (compare with the the Appendix) the collision set $Z_R : =x_R^{-1}(\Sigma)$ of $x_R$ has measure zero, $x_R \in \mathcal{C}^\infty([-\omega_R,\omega_R] \setminus Z_R;\mathbb{R}^3 \setminus \Sigma)$ and, for any 
	$t \in [-\omega_R,\omega_R] \setminus Z_R$, it holds that
	$$\ddot{x}_R(t) = \nabla V(x_R(t)),\qquad \frac{1}{2} \vert \dot{x}_R(t) \vert^2 = V(x_R(t)) + H.$$
Moreover, it can be shown that $x_R$ is indeed a true solution of \eqref{syst_xR}, i.e. $E_R=\emptyset$: when $\alpha>1$ this follows exactly as in \cite{BosDamTer17} using the Morse index bound (see also \cite{Tan93}); if $\alpha=1$ the occurrence of collisions can be excluded by a regularization argument (see Corollary \ref{cor-finalereg} and recall that $\xi^+\neq \xi^-$).

\smallbreak
Finally, as far as the level estimate \eqref{est_act} is concerned, we observe that the estimate from above can be obtained as in \cite[Sect. 4.4]{BosDamTer17}, while the estimate from below requires a sligthly different argument. To this aim, as a first step we notice that for any $u \in \Gamma_{R \xi^\pm}$ satisfying
\begin{equation}\label{min_leq_K}
\min_t \vert u(t) \vert \leq K,
\end{equation}
it holds that
\begin{equation}\label{ineq_maup}
\sqrt{\mathcal{M}_0(u)} \geq 2 \sqrt{H} (R - K).
\end{equation}
Indeed, from \eqref{min_leq_K} we deduce the existence of $t^-,t^+ \in (-1,1)$ such that $\vert u(t^\pm) \vert = K$ and $\vert u(t) \vert \geq K$ for $t \in [-1,t^-] \cup [t^+,1]$. Now, we introduce the notation
$$\mathcal{L}(u) = \int_{-1}^1 \vert \dot{u}(t) \vert \sqrt{V(u(t)) + H}\,dt.$$
Writing $r(t) = \vert u(t) \vert$, we obtain
\begin{align*}
\sqrt{\mathcal{M}_0(u)} \geq \mathcal{L}(u) & \geq \sqrt{H} \int_{-1}^1 \vert \dot{u}(t) \vert\,dt\\
& \geq \sqrt{H} \left( \int_{-1}^{t^-} \vert \dot{u}(t) \vert\,dt + \int_{t^+}^1 \vert \dot{u}(t) \vert\,dt \right)\\
& \geq \sqrt{H} \Big(\vert u(t^-) - u(-1) \vert + \vert u(1) - u(t^+) \vert \Big)\\
& \geq 2 \sqrt{H} (R - K).
\end{align*}
To conclude, we observe that the definition of the homotopy class $\Lambda_R$ implies that for any $h \in \Lambda_{R}$, there exists $s_h \in \mathbb{S}^1$ such that $h(s_h)([-1,1]) \cap [c_1,c_2] \neq \emptyset$; in particular, $h(s_h)$ satisfies \eqref{min_leq_K}. Hence
$$\sup_{s \in \mathbb{S}^1} \sqrt{\mathcal{M}_0(h(s))} \geq 2 \sqrt{H} (R - K)$$
and
\[
c_0\geq 2 \sqrt{H} (R - K).
\]
On the other hand, defining $u_R:[-1,1]\to \mathbb{R}^3$ by $u_R(t)=x_R(\omega_R t)$, for every $t\in [-1,1]$,
a simple argument (see \cite[Remark 4.7]{BosDamTer17}) shows that ${\mathcal M}_0(u_R)=c_0$, thus implying 
\[
{\mathcal M}_0(u_R)\geq 2 \sqrt{H} (R - K).
\]
Recalling the well-known equality
\[
\sqrt{\mathcal{M}_0(u_R)} = \frac{1}{\sqrt{2}} \mathcal{A}_{[-\omega_R,\omega_R]}(x_R),
\]
we deduce
$$\mathcal{A}_{[-\omega_R,\omega_R]}(x_R) \geq 2 \sqrt{2 H} (R - K),$$
as desired.

\section{Looking for entire solutions} \label{sec3}

In this section we show that the solution $x_R$ of \eqref{syst_xR} given in Theorem \ref{theo_act} converges, when $R\to +\infty$, to an entire hyperbolic solution of \eqref{main} with asymptotic directions $\xi^{\pm}$, thus proving Theorem \ref{main}.

\subsection{A preliminary result}\label{sec31}

We prove in the Lemma below the crucial property of the sequence of approximating solutions $x_R$: that is, the minimum of $|x_R|$ is bounded in $R$.

\begin{lemma}
	\label{lemma-minimo}
	Let $x_R$ be given by Theorem \ref{theo_act}; then
	\begin{equation} \label{eq-minimofinito}
	\limsup_{R\to +\infty} \min_t |x_R(t)|<+\infty.
	\end{equation}
\end{lemma}

\begin{proof}
Assume by contradiction that
$$\rho_R := \min_t \vert x_R(t) \vert = \vert x_R(\tau_R) \vert \to +\infty\quad \mbox{ as } R \to +\infty.$$
In particular, we can always suppose $\rho_R \geq K$; then, Lemma \ref{cor_t} is applicable and we obtain
\begin{equation}\label{est_omegaR}
\begin{aligned}
\omega_R - \tau_R & \geq \frac{R - \rho_R}{\sqrt{2(H + C_+/K^\alpha)}},\\
-\omega_R - \tau_R & \leq -\frac{R - \rho_R}{\sqrt{2(H + C_+/K^\alpha)}}.
\end{aligned}
\end{equation}
Let us set
$$d_R = \frac{\rho_R}{R} \in (0,1],\qquad d = \lim_{R \to +\infty} d_R \in [0,1],$$
and we distinguish two cases.
\medbreak
If $d = 0$, we define
$$v_R(t) = \frac{1}{\rho_R} x_R(\rho_R t + \tau_R),\quad t \in [\sigma_R^-,\sigma_R^+],$$
where
$$\sigma_R^\pm = \frac{\pm\omega_R - \tau_R}{\rho_R}.$$
Notice that $\vert v_R(0) \vert = 1$, $v_R(0) \cdot \dot{v}_R(0) = 0$ and $1 \leq \vert v_R(t) \vert \leq R/\rho_R$ for $t \in  [\sigma_R^-,\sigma_R^+]$. Writing $V$ as in \eqref{V_inf}, the function $v_R$ satisfies
$$\ddot{v}_R = -\frac{m v_R}{\rho_R^\alpha \vert v_R \vert^{\alpha+2}} + \rho_R \nabla W(\rho_R v_R)$$
and
$$\frac{1}{2} \vert \dot{v}_R \vert^2 = \frac{m}{\alpha \rho_R^\alpha \vert v_R \vert^\alpha} + W(\rho_R v_R) + H.$$
Moreover, from \eqref{est_omegaR} we obtain
\begin{align*}
\sigma_R^+ & \geq \frac{1/d_R - 1}{\sqrt{2 (H + C_+/K^\alpha)}} \to +\infty,\\
\sigma_R^- & \leq \frac{1 - 1/d_R}{\sqrt{2 (H + C_+/K^\alpha)}} \to -\infty,
\end{align*}
as $R \to +\infty$. Finally, using \eqref{K1} we find
\begin{equation}\label{CV_vR1}
\begin{aligned}
\vert \ddot{v}_R \vert & \leq \frac{m}{\rho_R^\alpha \vert v_R \vert^{\alpha+1}} + \rho_R \vert \nabla W(\rho_R v_R) \vert\\
& \leq \frac{m}{\rho_R^\alpha \vert v_R \vert^{\alpha+1}} + \frac{C_+}{\rho_R^{\alpha+2} \vert v_R \vert^{\alpha+3}}\\
& \leq \frac{m}{\rho_R^\alpha} + \frac{C_+}{\rho_R^{\alpha+2}} \to 0
\end{aligned}
\end{equation}
and
\begin{equation}\label{CV_vR2}
\begin{aligned}
\left\vert \frac{1}{2} \vert \dot{v}_R \vert^2 - H \right\vert & \leq \frac{m}{\alpha \rho_R^\alpha \vert v_R \vert^\alpha} + \vert W(\rho_R v_R) \vert\\
& \leq \frac{m}{\alpha \rho_R^\alpha \vert v_R \vert^\alpha} + \frac{C_+}{\rho_R^{\alpha+2} \vert v_R \vert^{\alpha+2}}\\
& \leq \frac{m}{\alpha \rho_R^\alpha} + \frac{C_+}{\rho_R^{\alpha+2}} \to 0
\end{aligned}
\end{equation}
as $R \to +\infty$, uniformly in $t$. We can thus readily see
that $v_R \to v_\infty$ in $\mathcal{C}_\text{loc}^2(\mathbb{R})$, with $v_\infty$ an entire hyperbolic solution with energy $H$ of the problem $\ddot{v}_\infty = 0$. Therefore, if we denote $v_\infty(0) = e_1$, $\dot{v}_\infty(0) = \sqrt{2 H} e_2$ with $\vert e_1 \vert = \vert e_2 \vert = 1$ and $e_1 \cdot e_2 = 0$, we have $v_\infty(t) = e_1 + \sqrt{2 H} t e_2$. As a consequence, with the notation $s_{v_\infty} = v_\infty/\vert v_\infty \vert$, we obtain
\begin{equation}\label{lim_svinf}
s_{v_\infty}(t) = \frac{e_1 + \sqrt{2 H} t e_2}{\sqrt{1 + 2 H t^2}} \to \pm e_2\quad \mbox{ as } t \to \pm\infty.
\end{equation}
Then, for any $\varepsilon > 0$, we choose $t_\varepsilon > 0$ such that
$$\frac{C_1 + C_2}{2 H t_\varepsilon} < \frac{\varepsilon}{2},$$
where $C_1, C_2$ are given by Lemma \ref{lem_ids}. 
For any $t > t_\varepsilon$, the $\mathcal{C}_\text{loc}^2(\mathbb{R})$ convergence ensures that 
$\vert s_{v_R}(t) - s_{v_\infty}(t) \vert < \varepsilon/2$ if $R$ is large enough. Using Lemma \ref{lem_ids} with $t_1 = \tau_R$, $\tau = \rho_R t_\varepsilon + \tau_R$ and $t_2 = \omega_R$, we have
\begin{align*}
\int_{t_\varepsilon}^{\sigma_R^+} \vert \dot{s}_{v_R}(t) \vert\,dt & = \int_{t_\varepsilon}^{\sigma_R^+} \vert \dot{s}_{x_R}(\rho_R t + \tau_R) \vert \rho_R\,dt = \int_{\rho_R t_\varepsilon + \tau_R}^{\omega_R} \vert \dot{s}_{x_R}(t) \vert\,dt\\
& \leq \frac{C_1 + C_2 \rho_R}{2 H \rho_R t_\varepsilon} \leq \frac{C_1 + C_2}{2 H t_\varepsilon} < \frac{\varepsilon}{2}.
\end{align*}
Therefore
\begin{align*}
\vert s_{v_\infty}(t) - \xi^+ \vert & \leq \vert s_{v_\infty}(t) - s_{v_R}(t) \vert + \vert s_{v_R}(t) - s_{v_R}(\sigma_R^+) \vert\\
& < \frac{\varepsilon}{2} + \int_t^{\sigma_R^+} \vert \dot{s}_{v_R}(\sigma) \vert\,d\sigma \leq \frac{\varepsilon}{2} + \int_{t_\varepsilon}^{\sigma_R^+} \vert \dot{s}_{v_R}(\sigma) \vert\,d\sigma\\
& < \frac{\varepsilon}{2} + \frac{\varepsilon}{2} = \varepsilon.
\end{align*}
The limit as $t \to -\infty$ being analogous, we thus derive
$$\lim_{t \to \pm\infty} s_{v_\infty}(t) = \xi^\pm$$
and recalling \eqref{lim_svinf}, we find $\xi^+ = e_2 = -\xi^-$, contradiction.
\medbreak
We now focus on the case $d \in (0,1]$. Let us define
$$\tilde{v}_R(t) = \frac{1}{R} x_R(R t + \tau_R),\quad t \in [\tilde{\sigma}_R^-,\tilde{\sigma}_R^+],$$
where
$$\tilde{\sigma}_R^\pm = \frac{\pm\omega_R - \tau_R}{R}.$$
The function $\tilde{v}_R$ satisfies
$$\ddot{\tilde{v}}_R = -\frac{m \tilde{v}_R}{R^\alpha \vert \tilde{v}_R \vert^{\alpha+2}} + R \nabla W(R \tilde{v}_R)$$
and
$$\frac{1}{2} \vert \dot{\tilde{v}}_R \vert^2 = \frac{m}{\alpha R^\alpha \vert \tilde{v}_R \vert^\alpha} + W(R \tilde{v}_R) + H.$$
Moreover, $\vert \tilde{v}_R(0) \vert = d_R$, $\tilde{v}_R(0) \cdot \dot{\tilde{v}}_R(0) = 0$, $\tilde{v}_R(\tilde{\sigma}_R^\pm) = \xi^\pm$ and $d_R \leq \vert \tilde{v}_R(t) \vert \leq 1$ for $t \in  [\tilde{\sigma}_R^-,\tilde{\sigma}_R^+]$. Finally, similarly as in \eqref{CV_vR1} and \eqref{CV_vR2},
\begin{equation}\label{CV_vR3}
\begin{aligned}
\vert \ddot{\tilde{v}}_R \vert & \leq \frac{m}{R^\alpha (d/2)^{\alpha+1}} + \frac{C_+}{R^{\alpha+2} (d/2)^{\alpha+3}},\\
\left\vert \frac{1}{2} \vert \dot{\tilde{v}}_R \vert^2 - H \right\vert & \leq \frac{m}{\alpha R^\alpha (d/2)^\alpha} + \frac{C_+}{R^{\alpha+2} (d/2)^{\alpha+2}},
\end{aligned}
\end{equation}
for $R$ large enough.

We now claim that $\tilde{\sigma}_R^+ - \tilde{\sigma}_R^-$ is bounded away from zero. Indeed, if $\tilde{\sigma}_R^- \to 0^-$ and $\tilde{\sigma}_R^+ \to 0^+$, then from
\begin{equation}\label{lim_vR}
\xi^\pm = \tilde{v}_R(\tilde{\sigma}_R^\pm) = \tilde{v}_R(0) + \int_0^{\tilde{\sigma}_R^\pm} \dot{\tilde{v}}_R(t)\,dt,
\end{equation}
together with the fact that $\max_t \vert \dot{\tilde{v}}_R(t) \vert $ is bounded in $R$ in view of \eqref{CV_vR3}, we obtain $\tilde{v}_R(0) \to \xi^-$ and $\tilde{v}_R(0) \to \xi^+$, which is not possible since $\xi^+ \neq \xi^-$.

As a consequence, there exists a nontrivial interval $\tilde{I}_\infty = (\tilde{\sigma}_\infty^-,\tilde{\sigma}_\infty^+)$ such that $\tilde{v}_R \to \tilde{v}_\infty$ in $\mathcal{C}_\text{loc}^2(\tilde{I}_\infty)$; moreover, $d \leq \vert \tilde{v}_\infty(t) \vert \leq 1$ for $t \in \tilde{I}_\infty$ and $\tilde{v}_\infty$ is a hyperbolic solution with energy $H$ of the problem $\ddot{\tilde{v}}_\infty = 0$. 
Since $\tilde{v}_\infty$ is bounded, we deduce that the interval $\tilde I_\infty$ is bounded; passing to the limit in \eqref{lim_vR}, 
we thus have that $\tilde{v}_\infty$ is a hyperbolic solution with energy $H$ of the (free-time) fixed-endpoints problem
$$\left\{ \begin{array}{l} \vspace{0.1cm}
\ddot{\tilde{v}}_\infty = 0,\\
\tilde{v}_\infty(\tilde{\sigma}_\infty^\pm) = \xi^\pm.
\end{array} \right.$$
Again, we write $\tilde{v}_\infty(t) = d e_1 + \sqrt{2 H} t e_2$ with $\vert e_1 \vert = \vert e_2 \vert = 1$ and $e_1 \cdot e_2 = 0$; and with the endpoint conditions, we infer
$$\tilde{\sigma}_\infty^\pm = \pm\frac{\sqrt{1 - d^2}}{\sqrt{2 H}}.$$
\smallbreak
If $d = 1$, we have an immediate contradiction because $\tilde{\sigma}_\infty^\pm = 0$.
\smallbreak
If $d \in (0,1)$, using the fact that $x_R$ has energy $H$, we write
\begin{align*}
\mathcal{A}_{[-\omega_R,\omega_R]}(x_R) & = 2 \int_{-\omega_R}^{\omega_R} \Big(V(x_R(t)) + H \Big)\,dt = 2 R \int_{\tilde{\sigma}_\infty^-}^{\tilde{\sigma}_\infty^+} \Big(V(R \tilde{v}_R(s)) + H \Big)\,ds\\
& = 2 R \int_{\tilde{\sigma}_\infty^-}^{\tilde{\sigma}_\infty^+} \left(\frac{m}{\alpha R^\alpha \vert \tilde{v}_R(s) \vert^\alpha} + R W(R \tilde{v}_R(s)) + H \right)\,ds
\end{align*}
so that, using \eqref{CV_vR3},
$$\lim_{R \to +\infty} \frac{\mathcal{A}_{[-\omega_R,\omega_R]}(x_R)}{R} = 2 H \left(\tilde{\sigma}_\infty^+ - \tilde{\sigma}_\infty^- \right) < 2 \sqrt{2 H}.$$
On the other hand, using Theorem \ref{theo_act} we find
$$\lim_{R \to +\infty} \frac{\mathcal{A}_{[-\omega_R,\omega_R]}(x_R)}{R} \geq 2 \sqrt{2 H},$$
so that a contradiction is obtained.
\end{proof}

\subsection{Passing to the limit}\label{sec32}

We are now in position to prove that a suitable translate of $x_R$ converges to an entire hyperbolic solution of \eqref{main} having asymptotic directions $\xi^\pm$ at $\pm \infty$. 

\smallbreak
To this aim, let $r_R = \vert x_R \vert$ and $K_R = \min_t r_R(t)$. The discussion after \eqref{ineq_LJ} implies that:
\begin{itemize}
\item[i)] if $K_R \geq K$, the function $r_R$ has a unique minimum point $t_R$,
\item[ii)] if $K_R < K$, there exist two unique instants $t_R^-,t_R^+$ with 
$t_R^- < t^+_R$ such that $r_R(t_R^{\pm}) = K$, $r_R(t) \leq K$ for $t \in [t_R^-,t_R^+]$, 
$$
r_R(t) > K \quad \mbox{ and } \quad \dot r_R(t) \neq 0, \; \mbox{ for } t \notin [t_R^-,t_R^+].
$$
\end{itemize}
In the first case, we define $t^+_R=t^-_R=t_R$; also, we introduce the constant $\tilde{K}_R = \max \{K,K_R \}$ and we observe that \eqref{eq-minimofinito} guarantees the existence of $\tilde{K} \geq K$ such that $\tilde{K}_R \leq \tilde{K}$ for any (large) $R$. 

We finally define
\begin{equation}\label{xR_trans}
\tilde{x}_R(t) = x_R \left(t + \frac{t_R^- + t_R^+}{2} \right), \qquad t \in [\omega_R^-,\omega_R^+],
\end{equation}
where 
$$\omega_R^- = -\omega_R - t_R^- - \Delta_R,\quad 
\omega_R^+ = \omega_R - t_R^+ + \Delta_R,\quad 
\Delta_R = \frac{t_R^+ - t_R^-}{2}.$$
The reason for this time-translation is that the time spent by the function $\tilde x_R$ inside the ball of radius $K$ is now the symmetric interval around the origin $[-\Delta_R,\Delta_R]$. 
\smallbreak
Now, we split the proof in some steps.
\smallbreak
\noindent
\underline{Claim 1:} it holds that 
\begin{equation}\label{eq-limsup_tR}
\limsup_{R \to +\infty}\Delta_R< +\infty.
\end{equation}
\smallbreak
\noindent
To prove this, we first use the conservation of the energy to write $\mathcal{A}_{[-\omega_R,\omega_R]}(x_R)$ as follows:
\begin{align*}
\mathcal{A}_{[-\omega_R,\omega_R]}(x_R) & = \sqrt{2} \int_{-\omega_R}^{t_R^-} \vert \dot{x}_R(t) \vert \sqrt{V(x_R(t)) + H}\,dt + 2 \int_{t_R^-}^{t_R^+} \Big(V(x_R(t)) + H \Big)\,dt\\
&\quad + \sqrt{2} \int_{t_R^+}^{\omega_R} \vert \dot{x}_R(t) \vert \sqrt{V(x_R(t)) + H}\,dt.
\end{align*}
Now, we simply estimate
$$\int_{t_R^-}^{t_R^+} \Big(V(x_R(t)) + H \Big)\,dt \geq H \left(t_R^+ - t_R^- \right),$$
and, using the monotonicity of $r_R$ for $t \notin [t_R^-,t_R^+]$ together with \eqref{K4}, we have
\begin{align*}
\int_{t_R^+}^{\omega_R} \vert \dot{x}_R \vert \sqrt{V(x_R) + H}\,dt & \geq \int_{t_R^+}^{\omega_R} \sqrt{\frac{m}{\alpha r_R^\alpha} + H} \vert \dot{r}_R \vert\,dt - \int_{t_R^+}^{\omega_R} \frac{C_+}{r_R^{\alpha+2}} \vert \dot{r}_R \vert\,dt\\
& = \int_K^R \sqrt{\frac{m}{\alpha r^\alpha} + H}\,dr - \int_K^R \frac{C_+}{r^{\alpha+2}}\,dr\\
& \geq F_K(R) - \frac{C_+}{\alpha + 1}.
\end{align*}
The same estimate holds also for $\int_{-\omega_R}^{t_R^-} \vert \dot{x}_R \vert \sqrt{V(x_R) + H}\,dt$. So, summing up, we can see that
$$2 H \left(t_R^+ - t_R^- \right) \leq \mathcal{A}_{[-\omega_R,\omega_R]}(x_R) - 2 \sqrt{2} F_K(R) + \frac{2 \sqrt{2} C_+}{\alpha + 1}.$$
Recalling the estimate from above in \eqref{est_act}, we conclude.
\smallbreak
\noindent
\underline{Claim 2:} it holds that $\omega^+_R\to +\infty$ and $\omega^-_R\to -\infty$, as $R\to +\infty$.
\smallbreak
\noindent
It is immediate to see that this fact is proved if we show that
\begin{equation} \label{eq-tempiill}
\omega_R - t_R^+ \to +\infty,\quad -\omega_R - t_R^- \to -\infty,\quad R\to +\infty.
\end{equation}
This follows from Lemma \ref{cor_t} with $t_1 = t_R^+$ and $t_2 = \omega_R$: indeed, we have
$$\omega_R - t_R^+ \geq
\frac{R - \tilde{K}_R}{\sqrt{2 (H + C_+/K^\alpha)}} \geq
\frac{R - \tilde{K}}{\sqrt{2 (H + C_+/K^\alpha)}},$$
whence the conclusion (for $-\omega_R - t_R^-$ the argument is the same).
\smallbreak
\noindent
\underline{Claim 3:} there exists a $H^1_{\textnormal{loc}}$-function $x_\infty: \mathbb{R} \to \mathbb{R}^3$ such that, for $R \to +\infty$,
$$\tilde{x}_R \to x_\infty\quad \mbox{ weakly in } H^1_\text{loc}(\mathbb{R}).$$
\smallbreak
\noindent
To prove this, we first observe that the same argument used to prove \eqref{eq-limsup_tR} shows that $\int_{-\Delta_R}^{\Delta_R} \vert 
\dot{\tilde{x}}_R(t) \vert^2 \,dt$ is bounded.
From this, together with \eqref{eq-limsup_tR} itself and the fact that $\tilde r_R(t) \leq K$ for $t \in [-\Delta_R,\Delta_R]$, we infer that 
$\int_{-\Delta_R}^{\Delta_R} \vert \tilde x_R(t) \vert^2 + \vert \dot{\tilde{x}}_R(t) \vert^2 \,dt$ is bounded as well. 
Using moreover the fact that the three quantities
$$\vert \tilde{x}_R(\pm\Delta_R) \vert,\quad
\vert \dot{\tilde{x}}_R(\pm\Delta_R) \vert = \sqrt{2 (V(\tilde{x}_R(\pm\Delta_R)) + H)}
$$
and
$$
\max_{t \notin [-\Delta_R,\Delta_R]} \vert \ddot{\tilde{x}}_R(t) \vert = \max_{t \notin [-\Delta_R,\Delta_R]} \vert \nabla V(\tilde{x}_R(t)) \vert
$$
are obviously bounded, we conclude that $\tilde x_R$ is bounded in $H^1_{\textnormal{loc}}(\mathbb{R})$ and a standard compactness argument gives the conclusion. Notice that, by \eqref{eq-limsup_tR} again, 
$\vert x_\infty(t) \vert \to \infty$ for $t \to \pm \infty$.
\smallbreak
\noindent
\underline{Claim 4:} writing $x_\infty = r_\infty s_\infty$, it holds that
\begin{equation}\label{asym_sinf}
\lim_{t \to \pm\infty } s_\infty(t) = \xi^\pm.
\end{equation}
\smallbreak
\noindent
We prove only the limit relation for $t \to +\infty$ (the other being analogous). Let $\Delta>0$ be such that
\[
\Delta_R\leq \Delta, \quad \forall \ R>K.
\]
For any $\varepsilon > 0$, we fix a $t_\varepsilon > \Delta$ such that
$$\frac{C_1 + C_2 \tilde{K}}{2 H (t_\varepsilon - \Delta)} < \frac{\varepsilon}{2},$$
where the constants $C_1,C_2$ are the ones in Lemma \ref{lem_ids}. Using the usual notation $\tilde{x}_R = \tilde{r}_R \tilde{s}_R$, from Lemma \ref{lem_ids} with the choices $t_1 = \Delta_R$, $\tau = t_\varepsilon$ and $t_2 = \omega_R^+$, we have that
$$\int_{t_\varepsilon}^{\omega_R^+} \vert \dot{\tilde{s}}_R(t) \vert\,dt \leq \frac{C_1 + C_2 \tilde{K}_R}{2 H (t_\varepsilon - \Delta_R)} \leq \frac{C_1 + C_2 \tilde{K}}{2 H (t_\varepsilon - \Delta)} < \frac{\varepsilon}{2}$$
for $R$ large enough. We are now in position to conclude. Indeed, for any $t > t_\varepsilon$ let us take $R$ so large that $\vert \tilde{s}_R(t) - s_\infty(t) \vert < \varepsilon/2$ (following from the convergence in Claim 3). Then
\begin{align*}
\vert s_\infty(t) - \xi^+ \vert & \leq \vert s_\infty(t) - \tilde{s}_{R}(t) \vert + \vert \tilde{s}_{R}(t) - \tilde{s}_R(\omega_R^+) \vert\\
& < \frac{\varepsilon}{2} + \int_t^{\omega_R^+} \vert \dot{\tilde{s}}_R(\sigma) \vert\,d\sigma \leq \frac{\varepsilon}{2} + \int_{t_\varepsilon}^{\omega_R^+} \vert \dot{\tilde{s}}_R(\sigma) \vert\,d\sigma\\
& < \frac{\varepsilon}{2} + \frac{\varepsilon}{2} = \varepsilon, 
\end{align*}
thus proving \eqref{asym_sinf}.
\smallbreak
\noindent
\underline{Claim 5:} $x_\infty$ is collision-free, namely $x_\infty(t) \notin \Sigma$ for any $t \in \mathbb{R}$. 
\smallbreak
\noindent
To prove this, we distinguish two cases depending on the value of $\alpha$.

In the case $\alpha \in (1,2)$, we argue as in \cite[Sect. 5.2]{BosDamTer17}. To give a sketch (and assuming for instance that $x_\infty(t) = c_1$ for some $t$) one defines the 
function $v_R$ by
$$v_R(t) = \frac{1}{\delta_R} \left(x_R \left(\delta_R^{1+\alpha/2} t + \tau_R \right) - c_1 \right),\qquad t \in [\sigma_R^-,\sigma_R^+],$$
where
$$
\delta_R = \min_t \vert x_R(t) - c_1 \vert \quad \mbox{ and } \quad \sigma_R^\pm = \frac{\tau_R^\pm - \tau_R}{\delta_R^{1+\alpha/2}},$$
and, using \eqref{deltastar2}-\eqref{deltastar3}, proves that $v_R \to v_\infty$ for $R \to +\infty$ in 
$\mathcal{C}_\text{loc}^2(\mathbb{R})$, with $v_\infty$ an entire zero-energy solution of the problem
$$\ddot{v}_\infty = -\frac{m_1 v_\infty}{\vert v_\infty \vert^{\alpha+2}}.$$
The same arguments of \cite{BosDamTer17} can then be used to show that the above convergence forces the Morse index of $x_R$ to be greater than a quantity $i(\alpha)$ such that $i(\alpha) \geq 2$ when $\alpha > 1$. Therefore, a contradiction with \eqref{est_jA} is obtained.

In the case $\alpha = 1$, we use the arguments in the Appendix. Precisely, from Corollary \eqref{cor-finalereg} we know that $x_\infty$ must be a collision-reflection solution near any of its possible collisions. But this contradicts the global property of being an unbounded solution with different asymptotic directions $\xi^\pm$.

\section*{Appendix}

In this Appendix, we describe a strategy to investigate the behavior of ``generalized solutions'' to \eqref{main} (when $\alpha = 1$, this being the most delicate case), so as to eventually rule out the occurrence of collision. We do not claim any originality in the forthcoming results, which are probably well known by experts in Celestial Mechanics; however, we hope it can be of some interest to collect them in the present form, since no appropriate reference in the literature seems to exist.

Throughout this section, we deal with the perturbed Kepler equation
\begin{equation}\label{kep}
\ddot q = - \frac{\mu q}{\vert q \vert^3} + \nabla U(q),
\end{equation}
where $\mu > 0$ and $U$ is a $\mathcal{C}^\infty$-function defined on some open set $\Omega \subset \mathbb{R}^3$ containing the origin; we will be interested in solutions to \eqref{kep} possibly taking the value $q = 0$. Notice that \eqref{main} can be written in the above form, setting (for some $i=1,\ldots,N$) $q = x - c_i,$ $\mu=m_i$ and $\Omega = \mathbb{R}^3 \setminus \cup_{j\neq i}\{c_j - c_i\}$; of course, such a choice leads to investigations about solutions colliding with the centre $c_i$.

Following \cite{BahRab89} we call \emph{generalized solution} to \eqref{kep} a continuous function
$q: I \subset \mathbb{R} \to \Omega$ (with $I \subset \mathbb{R}$ an interval) such that:
\begin{itemize}
\item[-] the set $Z := q^{-1}(0)$ of collisions has zero measure, 
\item[-] on $I \setminus Z$, the function $q$ is of class $\mathcal{C}^\infty$ and solves equation \eqref{kep} therein,
\item[-] the energy is preserved through collisions, i.e., there exists $h \in \mathbb{R}$ such that
$$
\frac{1}{2}\vert \dot q(t) \vert^2 - \frac{\mu}{\vert q(t) \vert} - U(q(t)) = h
$$
for any $t \in I \setminus Z$.
\end{itemize}

This is a very weak notion of solution; in order to restrict the attention to ``physically meaningful'' solutions, the incoming and outgoing collision directions at $t_0 \in Z$, namely
$$
\lim_{t \to t_0^-} \frac{q(t)}{\vert q(t) \vert} \qquad \mbox{ and } \qquad \lim_{t \to t_0^+} \frac{q(t)}{\vert q(t) \vert}, 
$$
play a role. Indeed, roughly speaking, a solution can be considered physically meaningful (that is, from a mathematical point of view, converted to a solution of a suitable regularized equation) if and only if the collision directions coincide (compare with \cite{BosOrtZhaPP}, where the more general situation of a time-dependent perturbation $U(t,q)$ is also discussed, and the equality between the collision directions is indeed incorporated in the definition of generalized solution). Actually, for such solutions the behavior is very simple: they are just reflected back after collision. We give here below a proof of this fact; it is worth mentioning that our arguments just rely on the classical Sperling estimates \cite{Spe69}, thus avoiding typical three-dimensional regularization techniques (like Kustaanheimo-Stiefel one, see for instance \cite{Wal08}).

\begin{proposition}\label{dir}
Let $q: (-\varepsilon,\varepsilon) \to \Omega$ be a generalized solution to \eqref{kep} with $q^{-1}(0) = \{0\}$.
Assume further that the limit 
$$
\lim_{t \to 0} \frac{q(t)}{\vert q(t) \vert} 
$$
exists. Then, $q$ is a collision-reflection solution to \eqref{kep}, i.e., 
$$
q(t) = q(-t), \quad \mbox{ for every } t \in (0,\varepsilon). 
$$
\end{proposition}

\begin{proof}
As proved in \cite{Spe69}, it holds that
\begin{equation}\label{col1}
q(t) = \left( \frac92 \mu\right)^{1/3} \vert t \vert^{2/3} \,\xi + \mathcal{O}\left( \vert t \vert^{4/3}\right), \qquad t \to 0,
\end{equation}
and
\begin{equation}\label{col2}
\dot q(t) = \frac23 \left( \frac92 \mu\right)^{1/3} t ^{-1/3} \, \xi + \mathcal{O}\left( \vert t \vert^{1/3}\right), \qquad t \to 0,
\end{equation}
where $\xi = \lim_{t \to 0}\frac{q(t)}{\vert q(t) \vert}$.
Based on this, we first define the Sundman integral
$$
s(t) = \int_0^t \frac{d\tau}{\vert q(\tau) \vert}, \qquad t \in (-\varepsilon,\varepsilon), 
$$
and we set, for $s$ in a neighborhood of zero,
$$
u(s) = q(t(s)),
$$
where $t(s)$ denotes as usual the inverse of $s(t)$; incidentally, notice that
\begin{equation}\label{def-d}
t(s) = \int_0^s \vert u(\sigma) \vert \,d\sigma.
\end{equation}
Then, we further set
\begin{align*}
v(s) & = (\vert q \vert \dot q) \circ t (s)\\
w(s) & = \left(- \frac{\mu q}{\vert q \vert} + \langle q, \dot q \rangle \dot q \right) \circ t (s).
\end{align*} 
Notice that the function $z(s) = (u(s),v(s),w(s))$ is defined on a punctured neighborhood of zero and is smooth therein.
Elementary computations, using the differential equation and the energy relation, show that, for any $s \neq 0$, 
\begin{align*}
u'(s) & = (\vert q \vert \dot q) \circ t(s)\\
v'(s) & = \left(- \frac{\mu q}{\vert q \vert} + \langle q, \dot q \rangle \dot q + \vert q \vert^2 \nabla U(q)\right) \circ 
t (s) \\
w'(s) & = \Big[\langle q, \dot q \rangle \vert q \vert \nabla U(q) + \big(2h + 2U(q) + \langle q, \nabla U(q) \rangle \big) \vert q \vert \dot q \Big] \circ t(s).
\end{align*} 
Writing the right-hand sides in terms of $(u,v,w)$, we thus see that $z(s)$
satisfies the differential equation
$$z'(s) = F(z(s)),\qquad s \neq 0,$$
where the vector field $F = (F_1,F_2,F_3)$ is given by
\begin{align*}
F_1(z) & = v\\
F_2(z) & = w + \vert u \vert^2 \nabla U(u)\\
F_3(z) & = \langle u, v \rangle \nabla U(u) + \big(2h + 2U(u)+ \langle u, \nabla U(u) \rangle \big) v.
\end{align*}
On the other hand, \eqref{col1} and \eqref{col2} readily imply that $z(s)$ can be continuously extended to $s = 0$, with
$$
z(0) = \left(0,0,\left[ -\mu + \frac49 \left(\frac92 \mu \right)^{2/3}\right]\xi \right) =: z_0.
$$
Therefore, $z(s)$ turns out to be a local solution of the Cauchy problem
$$
z' = F(z), \qquad z(0) = z_0;
$$
since the vector field $F$ satisfies
\begin{align*}
F_1(u,-v,w) & = -F_1(u,v,w)\\
F_2(u,-v,w) & = F_2(u,v,w)\\
F_3(u,-v,w) & = -F_3(u,v,w)
\end{align*}
we see that it has to be $u(s) = u(-s)$ for any $s$. Recalling \eqref{def-d}, we obtain $t(s) = -t(s)$, finally implying the conclusion.
\end{proof}

In view of Proposition \ref{dir}, it becomes of interest to investigate under which conditions generalized solutions actually have coincident incoming and outgoing collision direction. To present our results in this direction, we consider a sequence of \emph{classical} solutions $q_n: I \to \Omega$ of the equation
\begin{equation}\label{kepn}
\ddot q_n = - \frac{\mu q_n}{\vert q_n \vert^3} - \frac{2 \ve_n \mu q_n}{\vert q_n \vert^4} + \nabla U(q_n),
\end{equation}
where $\ve_n \geq 0$ and $\ve_n \to 0^+$; we also assume that the associated energy is independent of $n$, namely, there exists $h \in \mathbb{R}$ such that
\begin{equation}\label{kepen}
\frac{1}{2}\vert \dot q_n(t) \vert^2 - \frac{\mu}{\vert q_n(t) \vert} -  \frac{\ve_n \mu}{\vert q_n(t) \vert^2} - U(q_n(t)) \equiv h, 
\quad \mbox{ for any } n \geq 1.
\end{equation}
Of course, for $\ve_n = 0$ we are simply considering a family of classical solutions of the perturbed Kepler problem \eqref{kep}. On the other hand, the choice $\ve_n > 0$ allows us to deal with the solutions given by Proposition (notice indeed $q_n := x_{\beta_n} - c_i$ satisfies an equation like \eqref{kepn} in a sufficiently small neighborhoof of the origin). From now on, we will actually consider generalized solutions to \eqref{kep} arising as limits (in a suitable topology) of the above solutions $q_n$. We start with the following preliminary result.

\begin{lemma}\label{lem_pre}
Let $q_\infty: I \to \Omega$ be an $H^1$-function such that $q_n \to q_\infty$ weakly in $H^1(I)$. Then, $q_\infty$ is a generalized solution to \eqref{kep} and $q_\infty^{-1}(0)$ is a finite set. 
\end{lemma}

\begin{proof}
To prove that $q_\infty$ is a generalized solution to \eqref{kep} we just need to check that $q_\infty^{-1}(0)$ has zero measure. To this end, we use the energy relation \eqref{kepen} together with Fatou's lemma to write
\begin{align*}
\int_I \frac{\mu}{\vert q_\infty(t) \vert} \,dt & \leq \liminf_{n \to +\infty} \int_I \frac{\mu}{\vert q_n(t) \vert} \\
& \leq \liminf_{n \to +\infty} \int_I \left(\frac{1}{2}\vert \dot q_n(t) \vert^2 -  \frac{\ve_n \mu}{\vert q_n(t) \vert^2} - U(q_n(t)) - h \right) \,dt.
\end{align*}
Since $U(q_n)$ is $L^\infty$-bounded by uniform convergence and $\int_I \vert \dot q_n \vert^2$ is bounded by weak $H^1$-convergence, we see that the above quantity is finite, thus implying the conclusion. 

To prove that $q_\infty^{-1}(0)$ is a actually a finite set, we are going to show 
that collisions are isolated, namely, if $q_\infty(t_0) = 0$ for some $t_0 \in I$ then $q_\infty(t) \neq 0$ in a suitable punctured neighborhood of $t_0$. To this end, we preliminarily fix $r_0 > 0$ such that
$$
\frac{\mu}{\vert q \vert} + 2U(q) + \langle q, \nabla U(q) \rangle + 2h > 0, \quad \mbox{ for every } \vert q \vert \leq r_0.
$$
In view of the uniform convergence of $q_n$ to $q_\infty$, there exists an interval $J \subset I$ with $t_0 \in J$ such that
$\min_{t \in J} \vert q_n(t) \vert \leq r_0$ for any $n$ large enough. An easy computation shows that
$$
\frac{d^2}{dt^2} \frac12 \vert \dot q_n(t) \vert^2 = \frac{\mu}{\vert q_n(t) \vert} + 2U(q_n(t)) + \langle q_n(t), \nabla U(q_n(t)) \rangle + 2h
$$
for any $n$ large enough and for any $t \in J$, implying that the function $t \mapsto \vert q_n(t) \vert^2$ is strictly convex on $J$ for large $n$. Again by uniform convergence, $t \mapsto \vert q_\infty(t) \vert^2$ is convex on $J$, showing that either
$q_\infty \equiv 0$ on a neighborhood of $t_0$ or $q_\infty(t) \neq 0$ for every $t \in J \setminus \{t_0\}$. Since the first case cannot occur because the set of collisions has zero measure, we are done.
\end{proof}

\begin{remark}
It is worth mentioning that the choice of exponent $\beta = 2$ for the penalization term $\ve_n \vert q_n \vert^{-\beta}$ is crucial,
this being the unique value such that both a strong force assumption is satisfied (this being needed for the min-max argument of Section
\ref{sec2}) and the above Lagrange-Jacobi argument is possible.  
\end{remark}

In view of Lemma \ref{lem_pre}, we can perform a local analysis around each singularity. 
The next result is essentially proved in \cite{Tan94} and provides an estimate for the collision directions of a generalized solution $q_\infty$ in terms of the sequence $q_n$.

\begin{proposition}\label{dirtan}
Let $q_\infty: I \to \Omega$ be an $H^1$-function with $q_\infty^{-1}(0) =  \{0\}$ and such that $q_n \to q_\infty$ weakly in $H^1(I)$.
Then, if the limit 
$$
d:= \lim_{n \to +\infty} \frac{\ve_n }{\mu^{1/3} \min_t \vert q_n(t) \vert}
$$
exists finite, the angle between
$$
\lim_{t \to 0^-} \frac{q_\infty(t)}{\vert q_\infty(t) \vert} \qquad \mbox{ and } \qquad \lim_{t \to 0^+} \frac{q_\infty(t)}{\vert q_\infty(t) \vert} 
$$
is equal to $2\pi \sqrt{1+ d}$.
\end{proposition}

\begin{proof}[Sketch of the proof]
We argue as in the proof of \cite[Theorem 0.1 (ii)]{Tan94}, using a blow-up technique. Assuming $\delta_n := \min_t \vert q_n(t) \vert = \vert q_n(t_n) \vert$ with $t_n \to 0$, we set
$$
y_n(t) = \frac{1}{\delta_n} q_n\left( \delta_n^{3/2} \,t + t_n \right);
$$ 
then it is not difficult to see that $y_n$ converges uniformly on compact sets to a zero-energy solution $y_{\infty,d}$ of the equation
$$
\ddot y_{\infty,d} = - \frac{\mu y_{\infty,d}}{\vert y_{\infty,d} \vert^3} - \frac{2d \mu^{4/3} y_{\infty,d} }{\vert y_{\infty,d} \vert^4}.
$$
Via some delicate angular momentum estimates, it is possible to show that the collision directions of $q_\infty$ at $t = 0$ can be related to the asymptotic directions of $y_{\infty,d}$ (see \cite[Proposition 1.2]{Tan94}), which in turn are easily computed (see \cite[Proposition 1.1 (iii)]{Tan94}). Notice that all the assumptions in \cite{Tan94} are satisfied in our case, at least in a neighborhood of $q = 0$; even more, some simplifications are here possible with respect to the proof given therein, since we deal with fixed-energy solutions of an autonomous problem. However, the complete argument is still very long and we omit it.
\end{proof}

Combining Proposition \ref{dir} and Proposition \ref{dirtan} clearly suggests a strategy to exclude the occurrence of collisions for a generalized solution: indeed, whenever $d = 0$ the incoming and outgoing collisions directions must coincide so that $q_\infty$ is just a collision-reflection solution (a case which is typically ruled out for some other reasons). We end this appendix with a result presenting two cases (both used in the paper) in which the whole procedure works. In the first one, we simply deal with perturbations of the Kepler problem 
(that is, $\ve_n = 0$). In the second one, a Morse index assumptions is used; in the statement below, by Morse index $\textnormal{j}(q_n)$ of the solution $q_n$ we will mean the Morse index of $q_n$ when regarded as a critical point of the action functional on the space of $H^1$-paths with fixed ends.

\begin{corollary} \label{cor-finalereg}
Let $q_\infty: I \to \Omega$ be an $H^1$-function with $q_\infty^{-1}(0) =  \{0\}$ and such that $q_n \to q_\infty$ weakly in $H^1(I)$.
Then:
\begin{itemize}
\item[-] if $\ve_n = 0$ for every $n$, then $q_\infty$ is a collision-reflection solution,
\item[-] if $\textnormal{j}(q_n) \leq 1$ for every $n$, then $q_\infty$ is a collision-reflection solution.
\end{itemize}
\end{corollary}

\begin{proof}
The first case is obvious, since $\ve_n = 0$ for every $n$ clearly implies $d = 0$. The second case follows
from \cite[Proposition 1.1 (iv)]{Tan94}, since $d > 0$ would imply $\textnormal{j}(q_n) \geq 2$ for large $n$ (some care is needed since
in that paper a periodic boundary value problem is considered; however the constructed variations have compact support so
that the argument fits with our setting as well).
\end{proof}

\end{document}